\documentclass[12pt]{article}
\usepackage{no-ejc}
\usepackage{ytableau}
\usepackage{tikz}
\usetikzlibrary{matrix}

\usepackage{amsmath,amssymb}
\usepackage{graphicx}
\usepackage{enumerate}   
\usepackage[colorlinks=true,citecolor=black,linkcolor=black,urlcolor=blue]{hyperref}
\usepackage{cancel}

\dateline{}{}{}


\usepackage{hyperref}
\hypersetup{
    colorlinks,
    citecolor=black,
    filecolor=black,
    linkcolor=black,
    urlcolor=black
}

\usepackage{tikz}
\usetikzlibrary{arrows}
\usetikzlibrary{shapes,snakes}
\usepackage{geometry}
\makeatletter
\def\hlinewd#1{%
  \noalign{\ifnum0=`}\fi\hrule \@height #1 \futurelet
   \reserved@a\@xhline}
\makeatother

\Copyright{  The authors. Released under the CC BY-ND license (International 4.0).}

\title{Simple Algebraic Proofs of Uniqueness for Erd\H{o}s-Ko-Rado Theorems}

\author{
	Yuval Filmus\\
	\small Department of Computer Science\\[-0.8ex]
	\small Technion \\[-0.8ex] 
	\small\tt yuvalfi@cs.technion.ac.il
	
	\and
	
	Nathan Lindzey\\
	\small Department of Computer Science\\[-0.8ex]
	\small University of Colorado\\[-0.8ex] 
	\small\tt Nathan.Lindzey@colorado.edu
}

\begin{document}
\maketitle

\begin{abstract}
	We give simpler algebraic proofs of \emph{uniqueness} for several Erd\H{o}s-Ko-Rado results, i.e., that the canonically intersecting families are the only largest intersecting families. Using these techniques, we characterize the largest partially 2-intersecting families of perfect hypermatchings, resolving a recent conjecture of Meagher, Shirazi, and Stevens.
\end{abstract}

\section{Introduction}\label{sec:intro}

Let $\mathcal{X}$ be the collection of $k$-sets of an $n$-element set $[n] := \{1,2,\cdots,n\} =: \Sigma$. A family $\mathcal{F} \subseteq \mathcal{X}$ is \emph{intersecting} if $S \cap T \neq \emptyset$ for all $S,T \in \mathcal{F}$.  A family $\mathcal{F}$ is \emph{canonically intersecting} if there exists an $i \in \Sigma$ such that $\mathcal{F} = \mathcal{F}_i := \{ S \in \mathcal{X} : i \in S\}$. In \cite{ErdosKR61}, Erd\H{o}s, Ko, and Rado characterized the largest intersecting families of $\mathcal{X}$ for all $n$ and $k$ such that $k<n/2$.

\begin{theorem}[The Erd\H{o}s--Ko--Rado Theorem]\label{thm:EKR}
	Let $k < n/2$. If $\mathcal{F} \subseteq \mathcal{X}$ is intersecting, then
	\[
		|\mathcal{F}| \leq \binom{n-1}{k-1}.	
	\]
	Moreover, equality holds if and only if $\mathcal{F}$ is canonically intersecting.
\end{theorem}
\noindent We note for $k > n/2$ that any family is intersecting and for $k  = n/2$ that the characterization of the extremal families above does not hold.

There are many proofs of the Erd\H{o}s--Ko--Rado theorem, and it has since been generalized to a variety of combinatorial domains (e.g., words~\cite{Moon82}, subspaces~\cite{FranklW86}, permutations~\cite{CameronK03}). These results are collectively known as \emph{Erd\H{o}s-Ko-Rado (EKR) combinatorics}. Algebraic techniques have played a distinguished role in EKR combinatorics, so much so that a textbook has been written on the subject~\cite{GodsilMeagher}. We overview the modus operandi of such proofs.

Typically one starts by constructing a graph $\Gamma = (\mathcal{X},E)$ on a domain of objects $\mathcal{X}$ defined such that $xy \in E$ if $x$ and $y$ do not intersect. Recall that a set $S \subseteq \mathcal{X}$ is \emph{independent} if $xy \notin E$ for all $x,y \in S$. By design, the independent sets of $\Gamma$ are intersecting families; therefore, to deduce an EKR result for intersecting families of $\mathcal{X}$, it suffices to characterize the maximum independent sets of $\Gamma$. At first it is not clear how useful this graphical reformulation is until one considers the Delsarte--Hoffman \emph{ratio bound}, stated below. Recall that a \emph{pseudo-adjacency matrix} $A$ of a regular graph $G = (V,E)$ is a $|V| \times |V|$ matrix with constant row sum such that $ij \notin E \Rightarrow A_{ij} = 0$.
\begin{theorem}[Ratio Bound]
Let $A$ be a pseudo-adjacency matrix of a regular graph $G = (V,E)$ with eigenvalues $\lambda_{\max} \geq  \cdots \geq \lambda_{\min}$. If $S \subseteq V$ is independent, then
\[
|S| \leq  |V| \frac{-\lambda_{\min}}{\lambda_{\max} - \lambda_{\min}}.
\]
If equality holds, then the characteristic vector $1_S \in \mathbb{R}^V$ of $S$ lies in the span of the greatest and least eigenspace of $A$.
\end{theorem} 
\noindent Let $\Sigma$ be the universe of atoms that compose the members of $\mathcal{X}$ so that $x \subseteq \Sigma$ for all $x \in \mathcal{X}$. Big intersecting families $\mathcal{F}_e \subseteq \mathcal{X}$ can be constructed by taking all $x \in \mathcal{X}$ that contain some fixed atom $e \in \Sigma$ so that any two members of the family intersect at $e$. These are the so-called \emph{canonically intersecting families}, also known as \emph{trivially intersecting families} or \emph{stars}.
Remarkably, it is often the case that $\Gamma$ admits a pseudo-adjacency matrix with eigenvalues such that the ratio bound equals the size of a canonically intersecting family, thus proving that the canonically intersecting families are indeed maximum intersecting families.

This is a powerful technique for proving the \emph{bound} of an EKR result, but in practice, often the most difficult step of these algebraic EKR proofs is not showing that the ratio bound is tight, but rather leveraging the spectral consequences of the case of equality in the ratio bound to show \emph{uniqueness} --- that the canonically intersecting families are the \emph{only} maximum intersecting families (see~\cite[pg. 22]{GodsilMeagher}, for example). We are aware of essentially three algebraic methods for showing uniqueness: the \emph{rank}, \emph{polyhedral}, and \emph{width} methods. 

The rank method was first introduced by Godsil and Newman~\cite{NewmanPhD}, then further developed by Godsil and Meagher~\cite{GodsilM09}. The method proceeds by considering a binary matrix $M$ whose columns are the characteristic vectors of the canonically intersecting families. One first shows that the column space of $M$ spans the greatest and least eigenspaces of $A$, i.e., that any maximum intersecting family is a linear combination of canonically intersecting families. Next, one identifies a suitable submatrix $M'$ whose columns form a basis for the column space of $M$. The last and most difficult step is showing that the only maximum intersecting families spanned by $M'$ are the columns of $M'$ itself. The latter often requires laborious calculations involving ranks of certain submatrices of $M'$ and ad-hoc arguments concerning the 0/1 structure of such submatrices. This technique is quite versatile in that it has been shown to work not only in combinatorial settings but also algebraic ones such as $PGL(n,q)$~\cite{Spiga19}. The trade off here is that the method seems to require lots of case analyses and ad-hoc arguments that result in longer more complicated proofs. 

Note that the set of rows of $M$, denoted as $rows(M)$, are the characteristic vectors $1_x \in \mathbb{R}^{\Sigma}$ for all $x \in \mathcal{X}$ where $\Sigma$ is the universe of atoms that compose the elements of $V$. Later, both Godsil~\cite{GodsilMeagher, RooneyPhD} and Ellis et al.~\cite{EllisFP11} independently observed that the polyhedral structure of the convex hull of $rows(M)$ places more constraints on the structure of maximum intersecting families, leading to simpler arguments than those employed in the rank method. The disadvantage of this approach is that it is typically too much to ask for an explicit defining system of the convex hull of a set of 0/1 vectors. Such questions flirt with NP-completeness and have led to flawed proofs of EKR-type results (see~\cite{Filmus17}). On the other hand, when a good defining system is known, the polyhedral method gives more concise and elegant proofs of uniqueness, albeit at the expense of passing to polyhedral theory.

The efficacy of these methods is well-known (see~\cite{GodsilMeagher}); however, they are all a bit indirect, as they employ geometric or ad-hoc techniques that are not spectral per se. Alternatively, one can deduce uniqueness by solving the harder problem of characterizing the subsets of cometric association schemes that have \emph{dual width} 1~(see \cite[Ch. 8.6]{GodsilMeagher} for more details). This is a purely spectral approach, but such proofs are substantially more involved than the previous methods and are less general insofar that they only apply to cometric association schemes. Ideally, one seeks a simple spectral proof technique for showing uniqueness that just depends on eigenvalues and eigenvectors, in the spirit of the proof of the bound.

We show that the largest intersecting families can indeed be characterized using a simple spectral technique. It is based off the trivial observation that there is always a member of an intersecting family that can be perturbed to lie outside the family, and that one can often choose this perturbation to be sufficiently local for combinatorial objects. This general principle underlies recent work on \emph{stability} versions of asymptotic EKR-type results (see~\cite{EllisKL16,Ellis11}, for example), but here we show that this principle can in fact lead to succinct proofs of uniqueness in the \emph{exact} regime. 

To demonstrate the versatility of this method, we give simpler proofs of uniqueness for many of the well-known EKR theorems cataloged in~\cite{GodsilMeagher}, and we also use it to prove the main conjecture of Meagher et al.~\cite{MeagherPartial} in Section~\ref{sec:hypermatchings}. The algebraic framework we use for describing the majority of these results is the \emph{theory of association schemes}, which we overview in Section~\ref{sec:prelims}; however, in Section~\ref{sec:hypermatchings} we venture beyond association schemes, which will assume a more technical representation-theoretic framework. For the latter, we assume familiarity with finite group representation theory, especially that of the symmetric group (see~\cite{CST,MacDonald95} for a detailed treatment of the relevant theory).

\section{Preliminaries}\label{sec:prelims}

Let $\mathcal{X}$ be a domain over a universe $\Sigma$ of atoms, so that $x \subseteq \Sigma$ for all $x \in \mathcal{X}$. 
\begin{definition}[Association Scheme]\label{def:assoc} Let $\mathcal{X}$ be a domain. An \emph{association scheme} over $\mathcal{X}$ is a set $\mathcal{A} = \{A_0,A_1,\dots,A_d\}$ of binary $|\mathcal{X}| \times |\mathcal{X}|$ matrices that satisfy the following axioms:
	\begin{enumerate}
		\item $A_0 = I$,
		\item $\sum_{i=0}^d A_i = J$ where $J$ is the all-ones matrix,
		\item $A_i^\top \in \mathcal{A}$ for each $i$,
		\item $A_iA_j = A_jA_i \in \mathbb{C}\text{-Span}(\mathcal{A}) =: \mathfrak{A}$
	\end{enumerate}
	where $\mathfrak{A}$ is the \emph{Bose-Mesner algebra} generated by the basis $\mathcal{A}$. 
\end{definition}
The non-identity elements of $\mathcal{A}$ are called \emph{associates}. If $A_i^\top = A_i$ for all $i$, then we say that $\mathcal{A}$ is a \emph{symmetric} association scheme. All the association schemes considered in this work are symmetric. Each associate $A_i \in \mathcal{A}$ can be viewed as the adjacency matrix of a regular graph, so we will often conflate the two objects. We let $v_i$ denote the \emph{valency} of $A_i$ as a regular graph, i.e., its degree.  If the graph of each associate is (strongly) connected, then the association scheme is \emph{primitive}. All the association schemes considered in this work are primitive. We say two elements $x,y \in \mathcal{X}$ are \emph{i-related} if $(A_i)_{x,y} = 1$. 

Since $\mathfrak{A}$ is a commutative matrix algebra over $\mathbb{C}$, it admits a dual basis of \emph{primitive matrix idempotents} $E_0,E_1,\dots,E_d$ such that $E_iE_j = \delta_{i,j}E_i$ and  $\sum_{i=0}^d E_i = I$. Thus we have
\begin{align}\label{eq:relations}
A_i = \sum_{j=0}^{d}P_i(j)E_j \quad \text{ and } \quad E_i = \frac{1}{|\mathcal{X}|} \sum_{j=0}^{d}Q_i(j)A_j
\end{align}
where $P_i(j)$ denotes the \emph{$j$-th eigenvalue} of the $i$-th associate and $Q_i(j)$ denotes the \emph{$j$-th dual eigenvalue} of the $i$-th idempotent. It is clear that each $E_i$ is simply the orthogonal projector onto the eigenspace $V_i$ for all $0 \leq i \leq d$. Let $m_i := \text{Tr}~E_i = \dim V_i$ denote the $i$-th \emph{multiplicity}, i.e., the dimension of $V_i$. The $(d+1) \times (d+1)$ change of basis matrices $P := (P_{j,i}) = (P_{i}(j))$, $Q := (Q_{j,i}) = (Q_{i}(j))$ between the primal and dual bases of $\mathfrak{A}$ are called the \emph{character table} and \emph{dual character table}, respectively.  

There are a number of different ways to order the eigenspaces $V_i$ of $\mathfrak{A}$ (each of which induces a different ordering of the character tables), but the most natural orderings are those consistent with the polynomial structure of functions over the domain. Here, we consider functions $f \in \mathbb{C}^{\mathcal{X}}$ as polynomials in the indeterminants $\bar{x} := \{x_e\}_{e \in \Sigma}$, and the \emph{degree} of $f$ is defined to be the least degree of any polynomial $p_f(\bar{x})$ that \emph{represents} $f$, i.e., $f(x) = p_f(x)$ for all $x \in \mathcal{X}$ (see~\cite{DafniFLLV21} for more details). For many domains, such orderings take $V_0$ to be the space of degree-0 functions, $V_1$ to be the space of pure degree-1 functions, and so on.  The canonically intersecting families $1_{\mathcal{F}_e} \equiv x_e$ are degree-1 functions, and a unifying theme in EKR combinatorics is that they span the space $V_0 \oplus V_1$ of \emph{linear functions} of $\mathcal{X}$. 

In light of this, the orthogonal projection $E_0 + E_1$ onto the space of linear functions plays a central role in this work. We have $E_0 := J/|\mathcal{X}|$ for any association scheme, and an expression for $E_1$ can be obtained by determining the dual eigenvalues $\{Q_1(j)\}_{j=0}^d$ of $E_1$. Because $Q$ is essentially the inverse of $P$, expressions for the dual eigenvalues can be obtained provided we know the primal eigenvalues, i.e., $Q_j(i) = m_jP_{i}(j)/v_i$. 
Fortunately, the eigenvalues of association schemes often coincide with coefficients of orthogonal polynomials that admit well-known explicit expressions (see~\cite{BannaiI84,Delsarte73}, for example). Moreover, when the domain of the association scheme is a finite group $G$, the associates of the \emph{conjugacy class scheme} of $G$ correspond to the conjugacy classes of $G$ and its character table $P$ is essentially the character table of $G$. This allows one to determine $P$ via the character theory of $G$ (see~\cite{GodsilMeagher}).

Because the character theory of association schemes is so well-understood, writing down a simplified expression for the entries of $E_1$ requires nothing more than basic algebra. The simplicity of these expressions for $E_1$ will in turn allow for simple combinatorial expressions for the entries of $E_1f$ for any 0/1 vector $f$, which will be central to our proofs. In Section~\ref{sec:subsets} we carry out these calculations for $E_1$ of the Johnson scheme, but we omit them from the other sections, as they are entirely analogous. We refer the reader to~\cite{GodsilMeagher} for more details on association schemes and their application to EKR combinatorics.

We conclude this section with some notation that will be useful throughout this work. For any $\mathcal{F} \subseteq \mathcal{X}$, let $1_\mathcal{F} \in \mathbb{R}^{\mathcal{X}}$ be \emph{the characteristic vector} of $\mathcal{F}$, i.e., $(1_\mathcal{F})_x = 1$ if $x \in \mathcal{F}$, 0 otherwise. For any set $X$ and atom $e$, we use the notation $X + e := X \cup \{e\}$ and $X - e := X \setminus \{e\}$. Given $\mathcal{F} \subseteq \mathcal{X}$, for any $e \in \Sigma$, we define the \emph{e-restriction} $\mathcal{F}\!\!\downarrow_{e}$ to be the members of $\mathcal{F}$ that contain $e$, i.e.,
\[
\mathcal{F}\!\! \downarrow_{e} := \{ x \in \mathcal{F} : e \in x\} \quad \text{for all } e \in \Sigma.
\]
For a given domain $\mathcal{X}$ and universe $\Sigma$, we say that $\mathcal{F}$ is \emph{canonically intersecting} if 
\[
	\mathcal{F} = \mathcal{F}_e := \mathcal{X} \!\! \downarrow_e = \{ x \in \mathcal{X} : e \in x\}  \quad \text{for some } e \in \Sigma.
\] 
\section{Subsets}\label{sec:subsets}

Let $\mathcal{X}$ be the collection of $k$-element subsets of $\Sigma = [n]$. The \emph{Johnson scheme} $\mathcal{J}(n,k)$ is defined over $\mathcal{X}$ such that $S,T \in \mathcal{X}$ are $i$-related if $|S \cap T| = k - i$.
\begin{proposition}\label{prop:johnson}\emph{\cite{Delsarte73}} Let $P$ be the character table of $\mathcal{J}(n,k)$. For all $0 \leq i,j \leq k$, we have
	\[
		P_i(j) = \sum_{r = 0}^i (-1)^{r} \binom{j}{r} \binom{k-j}{i-r} \binom{n-k-j}{i - r} =   \sum_{r = i}^k (-1)^{r-i+j} \binom{r}{i} \binom{n-2r}{k-r} \binom{n-r-j}{r - j}.
	\]
\end{proposition}
\begin{proposition}\label{prop:johnsonE1} Let $\mathcal{F} \subseteq \mathcal{X}$. For all $S \in \mathcal{X}$ we have
	\[
	(E_1 1_\mathcal{F})_S = \frac{1}{\binom{n-2}{k-1}} \left( \sum_{i \in S} |\mathcal{F} \!\!\downarrow_i\!\!| -  \frac{k^2}{n}|\mathcal{F}|\right). 
	\]
\end{proposition}
\begin{proof}
By the relations stated in (\ref{eq:relations}) and Proposition~\ref{prop:johnson}, we may write $E_1$ as follows:
\begin{align*}
E_1 &= \frac{1}{\binom{n}{k}}\sum_{i=0}^k Q_1(i) A_i\\ 
&= \frac{n-1}{\binom{n}{k}} \sum_{i=0}^k\frac{ P_i(1)}{\binom{k}{i}\binom{n-k}{i}} A_i\\
&= \frac{n-1}{\binom{n}{k}} \sum_{i=0}^k \frac{ 1 }{\binom{k}{i}\binom{n-k}{i} } \left( \sum_{r = 0}^i (-1)^{r} \binom{1}{r} \binom{k-1}{i-r} \binom{n-k-1}{i - r}\right)A_i\\
&=\frac{n-1}{\binom{n}{k}}\sum_{i=0}^k \frac{ 1 }{\binom{k}{i}\binom{n-k}{i} } \left(\binom{k-1}{i} \binom{n-k-1}{i} -
\binom{k-1}{i-1} \binom{n-k-1}{i-1} \right)A_i\\
&=\frac{n-1}{\binom{n}{k}} \sum_{i=0}^k \left( 1 -\frac{ni}{k(n-k) } \right)A_i\\
&=\frac{1}{\binom{n-2}{k-1}} \sum_{i=0}^k \left((k-i)- \frac{k^2}{n} \right)A_i.
\end{align*}
This shows that $(E_1)_{S,T} = \binom{n-2}{k-1}^{-1} (|S \cap T| - k^2/n)$ for all $S,T \in \mathcal{X}$, thus we have
\[
	(E_1 1_\mathcal{F})_S = \frac{1}{\binom{n-2}{k-1}} \left(  \sum_{T \in \mathcal{F}} |S \cap T| -  \frac{k^2}{n}\right) = \frac{1}{\binom{n-2}{k-1}} \left( \sum_{i \in S} |\mathcal{F} \!\!\downarrow_i\!\!| -  \frac{k^2}{n}|\mathcal{F}|\right), 
\]
where the last equality follows from double counting.
\end{proof}

\noindent Given the expression above for $E_1$, we are now in a position to give a remarkably short algebraic proof of the Erd\H{o}s--Ko--Rado theorem. The proof of the bound is well-known (see~\cite{Lovasz79}).\\

\noindent \textbf{Theorem 1} (restated).
	\emph{Let $k < n/2$. If $\mathcal{F} \subseteq \mathcal{X}$ is intersecting, then}
	$
	|\mathcal{F}| \leq \binom{n-1}{k-1}.	
	$
	\emph{Moreover, equality holds if and only if $\mathcal{F}$ is canonically intersecting, i.e., $\mathcal{F} = \mathcal{F}_i$ for some $i \in [n]$.}
\begin{proof}
The independent sets of $A_k \in \mathcal{J}(n,k)$ are precisely the intersecting families of $\mathcal{X}$. By Proposition~\ref{prop:johnson}, we have $P_k(i) = (-1)^i \binom{n-k-i}{k-i}$ for all $0 \leq i \leq k$. The ratio bound shows
\[
	|\mathcal{F}| \leq  \binom{n}{k} \frac{\binom{n-k-1}{k-1}}{\binom{n-k}{k} + \binom{n-k-1}{k-1}}  = \binom{n-1}{k-1} = |\mathcal{F}_i|,
\] 
thus $1_{\mathcal{F}} \in V_0 \oplus V_1$ for any maximum intersecting $\mathcal{F}$. We must show $\mathcal{F} = \mathcal{F}_i$ for some $i \in [n]$.

Let $P = E_0 + E_1$ and define $P_S := (P1_\mathcal{F})_S$ for all $S \in \mathcal{X}$. 
Since $A_1$ is connected, pick $S_1 \in \mathcal{F}$ and $S_0 \notin \mathcal{F}$ to be 1-related. Proposition~\ref{prop:johnsonE1} and the fact that $P_S = (1_\mathcal{F})_S$ imply
\[
	P_{S_1} - P_{S_0} = 
	 \frac{1}{\binom{n-2}{k-1}} \left( \sum_{\ell \in S_1 }|\mathcal{F} \!\! \downarrow_{\ell} \! | 
	-  \sum_{\ell \in S_0 }|\mathcal{F} \!\! \downarrow_{\ell} \! |\! \right)  = 1, \text{ thus } 	\sum_{\ell \in S_1} |\mathcal{F} \!\! \downarrow_{\ell} \! | - \sum_{\ell \in S_0} |\mathcal{F}\!\! \downarrow_{\ell} \! |  = \binom{n-2}{k-1}.
\]
Since $S_1$ and $S_0$ are 1-related, all but two terms cancel, i.e.,
$
  |\mathcal{F} \!\! \downarrow_{i} \!\! | -  |\mathcal{F}\!\! \downarrow_{j} \!\! | = \binom{n-2}{k-1}
$
for some $i,j \in \Sigma$. We also have that $|\mathcal{F} \!\! \downarrow_{j} \!\! | \geq \binom{n-2}{k-2}$; otherwise, $\mathcal{F} \!  \setminus \! \mathcal{F} \!\! \downarrow_{j}$ is an intersecting family of $\binom{[n]-j}{k}$ of size greater than $\binom{n-2}{k-1}$, a contradiction. We conclude $\mathcal{F} = \mathcal{F}_i$, since we now have
\[
  \binom{n-1}{k-1} \geq |\mathcal{F} \downarrow_{i} \! | \geq \binom{n-2}{k-1} + \binom{n-2}{k-2} = \binom{n-1}{k-1}. \qedhere
\]
\end{proof}

\section{Subspaces}

Let $\mathcal{X}$ be the set of $k$-dimensional subspaces of $\mathbb{F}_q^n$. Recall the \emph{$q$-bracket} $[n]_q := (1-q^n)/(1-q)$ and the \emph{$q$-factorial} $[n]_q! := [1]_q[2]_q \cdots [n]_q$. It is well-known that $|\mathcal{X}|$ is given by the \emph{Gaussian binomial coefficient} ${ n \brack k}_q$, i.e.,
\[
	|\mathcal{X}| = \frac{[n]_q!}{[n-k]_q![k]_q!} =  {n \brack k}_q = {n-1 \brack k}_q + q^{n-k}{n-1 \brack k-1}_q.
\]  
We suppress the subscript from the brackets when $q$ is clear from context. We take $\Sigma$ to be the 1-dimensional subspaces of $\mathbb{F}_q^n$. The \emph{Grassmann scheme} $\mathcal{J}_q(n,k)$ is defined over $\mathcal{X}$ such that $S,T \in \mathcal{X}$ are $i$-related if $\dim S \cap T = k - i$. 
\begin{proposition}\label{prop:grassmann}\emph{\cite{Delsarte73}} Let $P$ be the character table of $\mathcal{J}_q(n,k)$. For all $0 \leq i,j, \leq k$, we have
\begin{align*}
P_i(j) &= \sum_{r=i}^k (-1)^{r-i+j} q^{\binom{r-i}{2} + \binom{j}{2} +r(k-j)} {r \brack i} {k-j \brack r-j} {n-r-j \brack k-j}\\
&=  \sum_{r=0}^j (-1)^r q^{\binom{r}{2} - ri} { {j \brack r} {n+1-r \brack r}} { i \brack r}. 
\end{align*}
\end{proposition}

\begin{proposition}\label{prop:grassmannE1}
	Let $\mathcal{F} \subseteq \mathcal{X}$. For all $S \in \mathcal{X}$ we have
	\[
	(E_1 1_\mathcal{F})_S  = \frac{1}{q^{k-1}{n-2 \brack k-1}} \left( \sum_{\ell \in S} |\mathcal{F} \!\!\downarrow_\ell\!\!| -  \frac{[k]^2}{[n]}|\mathcal{F}|\right). 
	\]
\end{proposition}
\noindent Two subspaces $S,T \in \mathcal{X}$ are \emph{skew} if $\dim S \cap T = 0$. The following proposition is a special case of~\cite[Lemma 9.3.1]{GodsilMeagher} that is not difficult to show.
\begin{proposition}\label{prop:skew}
	If $U$ is a $(k-1)$-dimensional subspace of a $k$-dimensional vector space $S$ over $\mathbb{F}_q$, then the number of 1-dimensional subspaces of $S$ skew to $U$ is $q^{k-1}$.
\end{proposition}
\noindent We say $\mathcal{F} \subseteq \mathcal{X}$ is \emph{intersecting} if no two of its members are skew. 
We are now ready to give a short proof of a $q$-analogue of Theorem~\ref{thm:EKR}. The proof of the bound is well-known (see~\cite{FranklW86}). 
\begin{theorem}\emph{\cite{FranklW86}}
	Let $k < n/2$. If $\mathcal{F} \subseteq \mathcal{X}$ is intersecting, then $|\mathcal{F}| \leq {n-1 \brack k-1}$.	
	Moreover, equality holds if and only if $\mathcal{F}$ is canonically intersecting.
\end{theorem}
\begin{proof}
The independent sets of $A_k \in \mathcal{J}_q(n,k)$ are precisely the intersecting families of $\mathcal{X}$. By Proposition~\ref{prop:grassmann}, we have $P_k(i) = (-1)^i q^{\binom{i}{2}+k(k-i)} {n-k-i \brack k-i}$ for all $i$. By the ratio bound
\[
		|\mathcal{F}| \leq  {n \brack k} \frac{q^{k(k-1)}{n-k-1 \brack k-1}}{q^{k^2}{n-k\brack k} + q^{k(k-1)} {n-k-1 \brack k-1}}  \leq {n-1 \brack k-1} = |\mathcal{F}\!\!\downarrow_{\ell}\!\!|,
\]
thus $1_\mathcal{F} \in V_{0} \oplus V_{1}$ for any maximum intersecting $\mathcal{F}$. We must show $\mathcal{F} = \mathcal{F}_\ell$ for some $\ell \in \Sigma$.  

Let $P = E_0 + E_1$ be the orthogonal projection onto $V_0 \oplus V_1$. Define $P_S := (P1_\mathcal{F})_S$, which by Proposition~\ref{prop:grassmannE1} can be written as
\[
	P_S = \frac{|\mathcal{F}|}{{n \brack k}} + \frac{1}{q^{k-1}{n-2 \brack k-1}} \left( \sum_{\ell \in S }|\mathcal{F} \!\!\downarrow_{\ell}\!\! | - \frac{k^2}{n} |\mathcal{F}| \right).
\]
Since $A_1$ is connected, pick $S_1 \in \mathcal{F}$ and $S_0 \notin \mathcal{F}$ to be 1-related. Since $P_S = (1_\mathcal{F})_S$, we have
	\[
		P_{S_1} - P_{S_0} = 
		\frac{1}{q^{k-1}{n-2 \brack k-1}} \sum_{\ell \in S_1 }|\mathcal{F} \!\! \downarrow_{\ell}\! \!| 
	-  \frac{1}{q^{k-1}{n-2 \brack k-1}}  \sum_{\ell \in S_0 }|\mathcal{F} \!\! \downarrow_{\ell}\! \!|  = 1, 
	\]
which implies that
	\[
		\sum_{\ell \in S_1} |\mathcal{F} \!\! \downarrow_{\ell} \!\! | - \sum_{\ell \in S_0} |\mathcal{F}\!\! \downarrow_{\ell} \!\!|  = q^{k-1}{n-2 \brack k-1}.
	\]
Let $X_1,X_0 \subseteq \Sigma$ be the 1-dimensional subspaces of $S_1$, $S_0$ that are skew to $S_1 \cap S_0$. We have
	\[
		\sum_{\ell \in  X_1 } |\mathcal{F} \!\! \downarrow_{\ell} \!\!| - \sum_{\ell \in X_0} |\mathcal{F} \!\! \downarrow_{\ell} \!\! |  = q^{k-1}{n-2 \brack k-1}.
	\]
Since $S_1$ and $S_0$ are 1-related, we have $|X_1| = |X_0| = q^{k-1}$ by Proposition~\ref{prop:skew}. After dropping negative terms and averaging, there exists an $\ell \in \Sigma$ such that $$|\mathcal{F} \!\! \downarrow_{\ell} \! \!|  \geq 	{n-2 \brack k-1} = { n-1 \brack k-1} - q^{n-1-k} {n-2 \brack k-2}. $$
Suppose there exists an $S' \in \mathcal{F} \setminus \mathcal{F} \!\! \downarrow_{\ell}$. Note that $S'$ must intersect each member of $\mathcal{F}\!\! \downarrow_{\ell}$.  A routine calculation shows that the total number of $S \in \mathcal{X}$ that contain $\ell$ and intersect $S'$ equals ${ n-1 \brack k-1} - q^{k(k-1)} {n-k-1 \brack k-1}$ (see \cite[pg. 3]{BlokhuisBCFMPS10}, for example), thus
\[
	|\mathcal{F} \!\! \downarrow_{\ell}\!\!| \leq { n-1 \brack k-1} - q^{k(k-1)} {n-k-1 \brack k-1}.
\]
To obtain a contradiction, we prove the following claim.\\

\noindent \textbf{Claim 1:} For all $k < n/2$, we have $q^{k(k-1)} {n-k-1 \brack k-1} > q^{n-1-k} {n-2 \brack k-2} $.
\begin{proof}[Proof of Claim 1]
After some rearranging, the claim is equivalent to showing that
\[
	q^{k-1}(q^{n-1-k} -1) ~q\left( \prod_{i=2}^{k-1} \left( \frac{q^{n-k-i} - 1}{q^{n-k-i}} \right) \right) > (q^{k-1}-1)q^{n-1-k}.
\]
For all $k < n/2$, we have $q^{k-1}(q^{n-1-k} -1) > (q^{k-1}-1)q^{n-1-k}$, so it suffices to show that 
$$ f(n,k,q) := \prod_{i=2}^{k-1} \left( \frac{q^{n-k-i} - 1}{q^{n-k-i}} \right)  > 1/2 \geq 1/q \quad \text{ for all } k < n/2.$$
For any fixed $n,q$, we have that $f$ is strictly decreasing as $k$ increases. For any fixed $n,k$, we have that $f$ is strictly increasing as a $q$ increases. Thus a lower bound on $f$ is obtained by setting $k = \lceil n/2 \rceil - 1$, $q = 2$, and letting $n \rightarrow \infty$, i.e.,
\[ f(n,k,q) \geq \lim_{n \rightarrow \infty} (3/4)(7/8) \cdots ((2^{\lceil n/2 \rceil - 1}-1)/2^{\lceil n/2 \rceil - 1}).\]
After taking logs, it suffices to show that 
 $\sum_{n=2}^{\infty} \log(2^{n}-1) - \log(2^n) > -1$. By standard Taylor expansion estimates, we have $\log(x + 1)-\log(x) \leq 1/x$, and since $\sum_{n=2}^{\infty} 1/(2^n-1) < 1$, the proof of the claim follows.
\end{proof}
\noindent We conclude that $\mathcal{F} \setminus \mathcal{F} \!\! \downarrow_{\ell}$ is empty, thus $\mathcal{F} = \mathcal{F}_{\ell}$, as desired.
\end{proof}

\section{Words}

Let $q \geq 3$ throughout, and let $\mathcal{X} = [q]^n$ be the set of $n$-symbol words $w$ drawn from $[q]$ with $\Sigma = [n] \times [q]$, so that $w = \{(i,w(i))\}_{i=1}^n$. The \emph{Hamming scheme} $\mathcal{H}(n,q)$ is defined such that two words of $\mathcal{X}$ are $i$-related if they agree on exactly $n-i$ coordinate positions. 
\begin{proposition}\label{prop:hamming}\emph{\cite{Delsarte73}} Let $P$ be the character table of  $\mathcal{H}(n,q)$. For all $0 \leq i,j \leq n$, we have
	$$P_i(j) = \sum_{r = 0}^j (-1)^r (q-1)^{i-r} \binom{n-j}{i-r}\binom{j}{r}.$$
\end{proposition}
\begin{proposition}\label{prop:hammingE1}
For all $w \in \mathcal{X}$ and $\mathcal{F} \subseteq \mathcal{X}$ we have
	\[
	 (E_11_{\mathcal{F}})_w = \frac{1}{q^{n-1}} \sum_{i = 1}^n \left(|\mathcal{F} \!\!\downarrow_{(i,w(i))} \!\! |-\frac{n}{q}|\mathcal{F}| \right).
	\]
\end{proposition}

\noindent We say $\mathcal{F} \subseteq \mathcal{X}$ is \emph{intersecting} if for any $w,w' \in \mathcal{F}$ there exists an $i \in [n]$ such that $w(i) = w'(i)$. We are now ready to give a short proof of the EKR theorem for intersecting families of words. The proof of the bound is well-known (see~\cite{Moon82}).
\begin{theorem}\emph{\cite{Moon82}}
	If $\mathcal{F} \subseteq \mathcal{X}$ is intersecting, then $|\mathcal{F}| \leq q^{n-1}$. Moreover, equality holds if and only if $\mathcal{F}$ is canonically intersecting.
\end{theorem}
\begin{proof}
	The independent sets of $A_n \in \mathcal{H}(n,q)$ are precisely the intersecting families of $\mathcal{X}$. By Proposition~\ref{prop:hamming}, we have $P_n(i) = (-1)^j(q-1)^{n-j}$ for all $0 \leq i \leq n$. By the ratio bound,
	\[
	|\mathcal{F}| \leq  q^n \frac{(q-1)^{n-1}}{(q-1)^n + (q-1)^{n-1}}  = q^{n-1} = |\mathcal{F}_{(i,j)}|, \text{ thus }
	\] 
	$1_{\mathcal{F}} \in V_0 \oplus V_1$ for any maximum intersecting $\mathcal{F}$. We now show $\mathcal{F}$ is canonically intersecting. 
	
	Let $P = E_0 + E_1$ and $P_w := (P1_\mathcal{F})_w$ for all $w \in \mathcal{X}$. 
	Since $A_1$ is connected, pick $w_1 \in \mathcal{F}$ and $w_0 \notin \mathcal{F}$ to be 1-related. By Proposition~\ref{prop:hammingE1} and the fact that $P_w = (1_\mathcal{F})_w$, we have
	\[
	q^{n-1} \left( P_{w_1} - P_{w_0}\right) =  \sum_{i=1}^n |\mathcal{F} \!\! \downarrow_{(i,w_1(i))} \!\!| 
	-  \sum_{i=1}^n|\mathcal{F} \!\! \downarrow_{(i,w_0(i))} \!\! | = q^{n-1}.
	\]
	Since $w_1$ and $w_0$ are 1-related, all but two terms cancel, i.e.,
	$
	|\mathcal{F} \!\! \downarrow_{(i,j)} \!\! | -  |\mathcal{F}\!\! \downarrow_{(i,j')} \!\! | = q^{n-1}
	$. We conclude that $\mathcal{F} = \mathcal{F}_{(i,j)}$ since $ q^{n-1} = |\mathcal{F}| \geq |\mathcal{F} \!\! \downarrow_{(i,j)} \!\! | \geq q^{n-1}$.
\end{proof}
\noindent It is easy to see that the proof above actually gives a stronger result that is well-known.
\begin{corollary}
	The only families $\mathcal{F} \subseteq [q]^n$ such that $1_\mathcal{F} \in V_0 \oplus V_1$ and $|\mathcal{F}| = q^{n-1}$ are the canonically intersecting families. Moreover, if $|\mathcal{F}| < q^{n-1}$, then $1_\mathcal{F} \notin V_0 \oplus V_1$.
\end{corollary}
\noindent The proof of uniqueness above does not hold for $q = 2$ as the graph $A_n$ is a perfect matching with spectrum $\{\pm 1^{2^{n-1}}\}$, but it is trivial to characterize the extremal families in this case.

\section{Bilinear Forms}

Let $\mathcal{X} = \mathbb{F}^{m \times n}_q$ be the set of all $m \times n$ matrices with $m < n$ and entries in $\mathbb{F}_q$, which can be identified with the set of all $m$-dimensional subspaces of $\mathbb{F}_q^{m+n}$ that are skew to a fixed $n$-dimensional subspace $W \leq \mathbb{F}_q^{m+n}$ (see~\cite{Delsarte78} or~\cite[Ch. 9.11]{GodsilMeagher} for more details). Let $\Sigma$ be the set of 1-dimensional subspaces of $\mathbb{F}_q^{m+n}$ skew to $W$. The \emph{bilinear forms scheme} $\mathcal{B}_q(m,n)$ is defined such that $A,B \in \mathcal{X}$ (as matrices) are $i$-related if $\text{rank}(A-B) = i$. This association scheme can be seen as a $q$-analogue of the Hamming scheme.

\begin{proposition}\emph{\cite{Delsarte78}}\label{prop:bilinear} Let $P$ be the character table of $\mathcal{B}_q(m,n)$. For all $0 \leq i,j \leq m$, we have
	\[
		P_j(i) = \sum_{k = 0}^m (-1)^{j-k} q^{kn + \binom{j-k}{2}} {m-k \brack m-j } { m -i \brack k}\quad \text{and} \quad P_i(j) = \frac{v_i}{v_j}P_j(i).
	\]
\end{proposition}

\begin{proposition}\label{prop:bilinearE1} For any $\mathcal{F} \subseteq \mathcal{X}$ and $A \in \mathcal{X}$, we have
	\[
		(E_11_{\mathcal{F}})_A = \frac{1}{q^{(m-1)n}} \left( \sum_{\ell \in A} |\mathcal{F}\!\!\downarrow_\ell\!\!| - \frac{[m]}{q^n}|\mathcal{F}| \right). 
	\]
\end{proposition}
\noindent We say $\mathcal{F} \subseteq \mathcal{X}$ is \emph{intersecting} if $A-B$ is not full rank for all $A,B \in \mathcal{F}$ as matrices, or equivalently, $\dim A \cap B \neq 0$ for all $A,B \in \mathcal{F}$ as subspaces. Here, we are only able to show a weaker version of the EKR theorem for  bilinear forms for $m \leq \lceil n/2 \rceil$ rather than $m < n$.
\begin{theorem}\emph{\cite{Huang87}} Let $m \leq \lceil n/2 \rceil$. If $\mathcal{F} \subseteq \mathcal{X}$ is intersecting, then $|\mathcal{F}| \leq q^{(m-1)n}$. Moreover, equality holds if and only if $\mathcal{F}$ is canonically intersecting.
\end{theorem}
\begin{proof}
	The independent sets of $A_m \in \mathcal{B}_q(m,n)$ are precisely the intersecting families of $\mathcal{X}$. By Proposition~\ref{prop:bilinear}, we have $P_m(i) = (v_m/v_i)(-1)^i {m \brack i}$ for all $0 \leq i \leq m$. By the ratio bound 
	\[
		|\mathcal{F}| \leq q^{mn} \frac{[m]_qv_m/v_1}{ v_m + [m]_qv_m/v_1}  = q^{mn} \frac{[m]_q}{ (q^n-1)[m]_q + [m]_q}  = q^{(m-1)n} = |\mathcal{F}_{\ell}|,
	\] 
	thus $1_{\mathcal{F}} \in V_0 \oplus V_1$ for any maximum intersecting $\mathcal{F}$. We now show $\mathcal{F}$ is canonically intersecting. Let $P = E_0 + E_1$ and define $P_B := (P1_\mathcal{F})_B$ for all $B \in \mathcal{X}$.
	Since $A_1$ is connected, let $B_1 \in \mathcal{F}$ and $B_0 \notin \mathcal{F}$ be 1-related. Proposition~\ref{prop:bilinearE1} and the fact that $P_B = (1_\mathcal{F})_B$ implies
	\[
	q^{(m-1)n} \left( P_{B_1} - P_{B_0} \right) =  \sum_{\ell \in B_1} |\mathcal{F} \!\! \downarrow_{\ell } \! | 
	-  \sum_{\ell \in B_0} |\mathcal{F} \!\! \downarrow_{\ell } \! | =  q^{(m-1)n}.
	\]
	Since $B_1$ and $B_0$ are 1-related, by Proposition~\ref{prop:skew} we may write
	$$
	\sum_{\ell \in X_1} |\mathcal{F} \!\! \downarrow_{\ell} \!\! | -  \sum_{\ell' \in X_0}|\mathcal{F}\!\! \downarrow_{\ell'} \!\! | = q^{(m-1)n}
	$$
	where $X_1,X_0 \subseteq \Sigma$ are the set of 1-dimensional subspaces of $B_1$ and $B_0$ that are skew to the $(m-1)$-dimensional space $B_1 \cap B_0$. Dropping negative terms and averaging shows there exists an $\ell \in \Sigma$ such that
	$$
		|\mathcal{F} \!\! \downarrow_{\ell} \!\! |\geq q^{(m-1)(n-1)}.
	$$
	Suppose there exists a $B' \in \mathcal{F} \setminus \mathcal{F} \!\! \downarrow_{\ell}$. Since $\mathcal{F}$ is intersecting, each member of $\mathcal{F} \!\! \downarrow_{\ell}$ must intersect $B'$. The total number of $B \in \mathcal{X}$ that contain a fixed $\ell \in \Sigma$ and also intersect a fixed $B' \in \mathcal{X}$ that is skew to $\ell$ equals $q^{(m-1)n} \left( 1- \left(\prod_{i=1}^{m-1} 1- q^{i}/q^n \right)\right)$, see~\cite[Lemma 2.3]{GongLW17} for a proof. Since $m \leq \lceil n/2 \rceil$, using arguments similar to Claim 1, we have 
	\[
		 |\mathcal{F}\!\!\downarrow_{\ell}\!\!| \leq q^{(m-1)n} \left( 1- \left(\prod_{i=1}^{m-1} 1-\frac{q^{i}}{q^n} \right)\right) = q^{(m-1)n} \left( 1- \left(\prod_{i=1}^{m-1} \frac{q^{n-i}-1}{q^{n-i}} \right)\right)  < q^{(m-1)(n-1)}.
	\]
This contradicts the fact that $|\mathcal{F} \!\! \downarrow_{\ell} \!\! |\geq q^{(m-1)(n-1)}$, thus $\mathcal{F} \setminus \mathcal{F} \!\! \downarrow_{\ell}$ is empty, and so $\mathcal{F}$ must be canonically intersecting.
\end{proof}
\noindent It is likely that similar proofs hold for other sesquilinear forms (e.g., Hermitian, Alternating) whose characters are specializations of the \emph{generalized Krawtchouk polynomials}~\cite{Delsarte78}.

\section{Permutations}

Let $\mathcal{X}$ be the \emph{symmetric group} on $n$ symbols with $\Sigma = [n] \times [n]$, so that $\sigma = \{(i,\sigma(i))\}_{i=1}^n$ for all $\sigma \in \mathcal{X}$.
Recall for any irreducible representation $\rho$ of a finite group $G$ that the orthogonal projection $E_\rho$ onto the $\rho$-isotypic component of its group algebra can be written in matrix form as

\[
	(E_\rho)_{h,g} = \frac{\dim \rho}{|G|} \chi^{\rho}( gh^{-1}) \quad \text{ for all } h,g \in G
\]
where $\chi^\rho$ is the irreducible character of $\rho$. Basic character theory of $S_n$ shows that the orthogonal projection $E_1 := E_{(n-1,1)}$ onto the $(n-1,1)$-isotypic component $V_1$ is
\[
	(E_1)_{\sigma, \pi} = \frac{n-1}{n!} \chi^{(n-1,1)}( \pi \sigma^{-1}) = \frac{n-1}{n!}\left( \text{fp}( \pi \sigma^{-1})  -1 \right)\quad \text{ for all } \sigma,\pi \in \mathcal{X}
\]
where $\text{fp}(\sigma)$ is the number of fixed points of $\sigma$. The following proposition is now immediate.
\begin{proposition}\label{prop:permutationE1} For all $\mathcal{F} \subseteq \mathcal{X}$ and $\sigma \in \mathcal{X}$, we have 
\[
	(E_11_{\mathcal{F}})_{\sigma} = \frac{n-1}{n!} \left( \sum_{i=1}^n |\mathcal{F} \! \downarrow_{(i,\sigma(i))} \!| -  |\mathcal{F}| \right).
\]
\end{proposition}
\noindent We are now ready to prove the EKR theorem for intersecting families of permutations. 
\begin{theorem}
If $\mathcal{F} \subseteq \mathcal{X}$ is intersecting, then $|\mathcal{F}| \leq (n-1)!$. Moreover, equality holds if and only if $\mathcal{F}$ is canonically intersecting.
\end{theorem}

\begin{proof}
There are a few different proofs that $|\mathcal{F}| = (n-1)!$ and $1_{\mathcal{F}} \in V_0 \oplus V_1$ for any maximum intersecting family $\mathcal{F} \subseteq \mathcal{X}$ (see~\cite{Renteln07,GodsilM09,Ellis12}). We show $\mathcal{F}$ must be canonically intersecting.

 Let $P = E_0 + E_1$ and define $P_\sigma := (P1_\mathcal{F})_\sigma$ for all $\sigma \in \mathcal{X}$. The transpositions $\tau$ generate $\mathcal{X}$, so let $\sigma_1 \in \mathcal{F}$ and $\sigma_0 \notin \mathcal{F}$ such that $\sigma_0 = \tau\sigma_1$.  By Proposition~\ref{prop:permutationE1} we may write 
\[
	P_{\sigma_1} - P_{\sigma_0} = 
	\frac{n-1}{n!}\left( \sum_{i=1}^n |\mathcal{F} \! \downarrow_{(i,\sigma_1(i))} \!| - \sum_{i=1}^n |	\mathcal{F} \! \downarrow_{(i,\sigma_0(i))} \!| \right)  = 1,
\]
which implies that
\[
	\sum_{i=1}^n |\mathcal{F} \! \downarrow_{(i,\sigma_1(i))} \!|  -  \sum_{i=1}^n |	\mathcal{F} \! \downarrow_{(i,\sigma_0(i))} \!|  = \frac{n!}{n-1} = (n-1)! + (n-2)!.
\]
Without loss of generality, let $\sigma_1 = ()$ and $\sigma_0 = (1,2)$. Then we have
\[
	\sum_{i=1}^2 |\mathcal{F} \! \downarrow_{(i,\sigma_1(i))} \!|  -  \sum_{i=1}^2 |	\mathcal{F} \! \downarrow_{(i,\sigma_0(i))} \!|  = (n-1)! + (n-2)!.
\]
After dropping negative terms and averaging, without loss of generality, we have
\[
	|\mathcal{F} \! \downarrow_{(1,1)} \!|  \geq  \frac{(n-1)! + (n-2)!}{2}.
\]
To finish the proof, we use a trick of Ellis~\cite{Ellis12} to boost the lower bound on $|\mathcal{F} \! \downarrow_{(1,1)} \!|$. 

Two families $\mathcal{F},\mathcal{F}'$ are \emph{cross-intersecting} if $\sigma$ and $\sigma'$ intersect for all $\sigma \in \mathcal{F}$ and $\sigma' \in \mathcal{F}'$.
Note that $\mathcal{F} \! \downarrow_{(1,1)}$ and $\mathcal{F} \! \downarrow_{(1,i)}$ are cross-intersecting since $\mathcal{F}$ is intersecting. It follows from \cite[Lemma 3.4]{Ellis12} that  
\[
	|\mathcal{F} \! \downarrow_{(1,1)}| \cdot  |\mathcal{F} \! \downarrow_{(1,i)}| \leq ((n-2)!)^2.
\]
We deduce that  $|\mathcal{F} \! \downarrow_{(1,i)}| \leq 2(n-2)!/n < 2(n-2)!/(n-1)$ for all $i \neq 1$, which implies that 
$$
| \mathcal{F} \setminus \mathcal{F} \downarrow_{(1,1) }  \!| = \sum_{i=2}^n | \mathcal{F} \downarrow_{(1,i) }\!| \leq  2(n-2)!.
$$

The foregoing shows that for any intersecting $\mathcal{F}$ there exists an $i,j \in [n]$ such that
\begin{align}\label{eq:lbperm}
	|\mathcal{F} \!\! \downarrow_{(i,j) }\!| \geq (n-1)! - 2(n-2)!.
\end{align}
Suppose there exists a $\sigma' \in \mathcal{F} \setminus \mathcal{F} \!\! \downarrow_{(i,j)}$. Note that every $\sigma \in \mathcal{F} \!\!\downarrow_{ (i,j) }$ must intersect $\sigma'$. The number of permutations of $\mathcal{X}$ that do not intersect a fixed permutation is $\lfloor n!/e \rceil$, the number of \emph{derangements}. The total number of $\sigma \in \mathcal{X}$ such that $\sigma(i) = j$ and $\sigma(k) = \pi(k)$ for some $k \neq i$ is $(n-1)! - \lfloor (n-1)!/e \rceil - \lfloor (n-2)!/e \rceil$ (see~\cite{Ellis12}, for example), thus
$$|\mathcal{F} \downarrow_{ (i,j) }\!| \leq (n-1)! - \lfloor (n-1)!/e \rceil - \lfloor (n-2)!/e \rceil.$$ 
This contradicts (\ref{eq:lbperm}) for all $n \geq 6$. Since $\mathcal{F} \setminus \mathcal{F} \!\! \downarrow_{(i,j)}$ is empty, we have that $\mathcal{F}$ is a canonically intersecting family. For $n < 6$ one can verify the theorem above by a simple computer search, which completes the proof.
\end{proof}
\noindent 
\section{Perfect Matchings}
Let $\mathcal{X}$ be the set of $(2n-1)!!$ perfect matchings of the complete graph $K_{2n} = ([2n],E)$ with $\Sigma = E$. Let $\lambda \vdash n$ be an integer partition of $n$. Let $\ell(\lambda)$ denote the number of parts of $\lambda \vdash n$. The associates and primitive idempotents of the \emph{perfect matching scheme} $\mathcal{A}$ are indexed by the set of integer partitions $\lambda  \vdash n$ and are defined such that two perfect matchings $m,m' \in \mathcal{X}$ are $\lambda$-related if the multiunion $m \cup m'$ is multigraph-isomorphic to the multigraph $\sqcup_{i=1}^{\ell(\lambda)} C_{2\lambda_i}$ where $C_{2\lambda_i}$ is the cycle on $2\lambda_i$ edges.
We say $m,m' \in \mathcal{X}$ are 1-related if they are $(2,1^{n-2})$-related. Let $E_1$ be the orthogonal projection onto the $(n-1,1)$-eigenspace $V_1$. See~\cite[Ch.~15]{GodsilMeagher} for more details on the perfect matching association scheme.
\begin{proposition}\emph{\cite{Lindzey20}}\label{prop:matchingE1} For all $\mathcal{F} \subseteq \mathcal{X}$ and $m \in \mathcal{X}$, we have
\[
	(E_11_{\mathcal{F}})_m = \frac{1}{(2n-5)!! 2(n-1)}\left(\sum_{ij \in m} | \mathcal{F}\!\! \downarrow_{ij}\!\! | - \frac{|\mathcal{F}|}{2(n-1)}\right).
\]
\end{proposition}
\noindent It is well-known that the number of \emph{perfect matching derangements}, i.e., the number of $m \in \mathcal{X}$ that share no edges with some fixed $m' \in \mathcal{X}$, equals $\lfloor (2n-1)!!/\sqrt{e}~ \! \rceil$, which one may compare to permutation derangements. This count implies the following.
\begin{proposition}\label{prop:count} For any $ij \in \Sigma$ and $m' \in \mathcal{X}$ such that $ij \notin m'$, we have
	\[
		|\{ m \in \mathcal{X} : ij \in m \text{ and } m \cap m' \neq \emptyset \}| \leq (1-1/\sqrt{e})(2n-3)!! < \frac{2}{5}(2n-3)!!.
	\]
\end{proposition}
\noindent We are now in a position to prove Theorem~\ref{thm:hamiltonian}, a stronger result that is easily seen to imply the EKR theorem for intersecting families of perfect matchings of $K_{2n}$.
\begin{theorem}\label{thm:hamiltonian} \emph{\cite{Lindzey17}}
If $\mathcal{F} \subseteq \mathcal{X}$ is non-Hamiltonian, i.e., $m \cup m' \not \cong C_{2n}$ for all $m,m' \in \mathcal{F}$, then $|\mathcal{F}| \leq (2n-3)!!$. Equality holds if and only if $\mathcal{F}$ is a canonically intersecting family.
\end{theorem}
\begin{proof}
Independent sets of $A_n := A_{(n)} \in \mathcal{A}$ are non-Hamiltonian families of $\mathcal{X}$. In~\cite{Lindzey17}, it was shown that $P_{n}((n)) = (2n-2)!!$ and $P_n((n-1,1)) = -(2n-4)!!$ are the unique greatest and least eigenvalues of $A_n$ corresponding to $V_0 := V_{(n)}$ and $V_1 := V_{(n-1,1)}$. By the ratio bound
\[
	|\mathcal{F}| \leq (2n-1)!! \frac{(2n-4)!!}{(2n-2)!!+(2n-4)!!} = (2n-3)!! = |\mathcal{F}_{ij}|,
\]
thus $1_{\mathcal{F}} \in V_0 \oplus V_1$ for any maximum intersecting $\mathcal{F}$. We show that $\mathcal{F} = \mathcal{F}_{ij}$ for some $ij \in \Sigma$.

Let $P = E_0 + E_1$ and $P_m := (P1_\mathcal{F})_m$.
Since the graph of $A_1$ is connected, pick $m_1 \in \mathcal{F}$ and $m_0 \notin \mathcal{F}$ to be 1-related.  Since $P_m = (1_\mathcal{F})_m$ for all $m \in \mathcal{X}$, we may write 
\[
P_{m_1} - P_{m_0} = 
\frac{1}{(2n-5)!! 2(n-1)} \sum_{ij \in m_1} |\mathcal{F}\!\! \downarrow_{ij} \!\!| - \frac{1}{(2n-5)!! 2(n-1)} \sum_{ij \in m_0} |\mathcal{F}\!\! \downarrow_{ij} \!\!|  = 1, \text{ thus}
\]
\[
\sum_{ij \in m_1} |\mathcal{F}\!\! \downarrow_{ij}\!\! | - \sum_{ij \in m_0} |\mathcal{F}\!\! \downarrow_{ij} \!\!|  = (2n-5)!! 2(n-1) = (2n-3)!! + (2n-5)!!.
\]
Since $m_1$ and $m_0$ are 1-related, all but 4 terms cancel, so without loss of generality, we have
\[
|\mathcal{F}\!\! \downarrow_{ij}\!\! | + |\mathcal{F}\!\! \downarrow_{k\ell}\!\! | - |\mathcal{F}\!\! \downarrow_{ik}\!\! | - |\mathcal{F}\!\! \downarrow_{j\ell} \!\!|= (2n-3)!! + (2n-5)!!,
\]
Dropping negative terms and averaging shows, without loss of generality, that 
\[
|\mathcal{F} \!\! \downarrow_{ij} \!\!| \geq \frac{(2n-3)!! + (2n-5)!!}{2}.
\]
Suppose there exists an $m' \in \mathcal{F} \! \setminus \! \mathcal{F}\!\!\downarrow_{ij}$. Each $m \in \mathcal{F}\!\! \downarrow_{ij}$ must intersect $m'$, so by Proposition~\ref{prop:count}
\[
	\frac{(2n-3)!! + (2n-5)!!}{2} \leq |\mathcal{F}\!\! \downarrow_{ij}\!\! | < \frac{2}{5}(2n-3)!!, 
\]
which is a contradiction. We deduce that $\mathcal{F} \! \setminus \! \mathcal{F}\!\!\downarrow_{ij}$ is empty, thus $\mathcal{F} = \mathcal{F}_{ij}$, as desired.
\end{proof}

\section{Partially 2-Intersecting Perfect Hypermatchings}\label{sec:hypermatchings}
\noindent In this section we prove a hypergraph generalization of the characterization of the largest intersecting families of perfect matchings of $K_{2n}$ to so-called partially 2-intersecting \emph{perfect hypermatchings} of $k$-uniform hypergraphs on $kn$ vertices $K_{kn}^k = ([kn],E)$. Let $\mathcal{M}_{kn}^k$ be the set of perfect hypermatchings of $K_{kn}^k$. Two perfect hypermatchings $m,m' \in \mathcal{M}_{kn}^k$ are said to \emph{partially 2-intersect} if there exists a pair of $k$-edges $e \in m$, $e' \in m'$ such that $|e \cap e'| \geq 2$. Meagher et al.~\cite{MeagherPartial} show for sufficiently large $n$ that if $\mathcal{F} \subseteq \mathcal{M}_{kn}^k$ is partially 2-intersecting, then 
\[
	|\mathcal{F}| \leq \binom{kn-2}{k-2} |\mathcal{M}_{k(n-1)}^k|.
\]
The \emph{canonically partially 2-intersecting} families $\mathcal{F}_{ij} = \{m \in \mathcal{M}_{kn}^k : ij \subseteq e \in m\}$ where $ij \subseteq [kn]$ meet this bound with equality. Their proof uses the ratio bound, and since equality is met, a consequence of their proof is that any maximum partially 2-intersecting family lives in the direct sum of the greatest eigenspace $V_0$ and least eigenspace $V_1$ of the \emph{partial 2-derangement graph} $\Gamma = (\mathcal{M}_{kn}^k,E)$ defined such that $m,m'$ are adjacent if there exists no pair of edges $e \in m$, $e' \in m'$ such that $|e \cap e'| \geq 2$. In particular, $V_0$ and $V_1$ are the trivial module and the $S_{kn}$-irreducible corresponding to $(kn-2,2) \vdash kn$ respectively.

Meagher et al.~conjectured for sufficiently large $n$ that the canonically partially 2-intersecting families are the only largest partially 2-intersecting families. The main result of this section is Theorem~\ref{thm:partial} which proves their conjecture. The proof of this result is a bit more involved as we can no longer exploit off-the-shelf character theory to derive the projector $P = E_0 + E_1$ onto the space $V_0 \oplus V_1$ as we did in the previous sections; therefore, we must derive $P$ from scratch using representation-theoretical techniques. As always, we have $E_0 = J/|\mathcal{M}_{kn}^k|$, so the task at hand is deriving the orthogonal projector $E_1$ onto $V_1$. 

Let $\mathbb{R}[\mathcal{M}_{kn}^k]$ be the space of real-valued functions over $\mathcal{M}_{kn}^k$. 
The stabilizer subgroup of any fixed perfect hypermatching
is isomorphic to the wreath product $H_{k,n} := S_k \wr S_n \leq S_{kn}$, thus $\mathcal{M}_{kn}^k \cong S_{kn}/H_{k,n}$. The multiplicities of the permutation representation $1\!\!\uparrow_{H_{k,n}}^{S_{kn}} \cong \mathbb{R}[\mathcal{M}_{kn}^k]$ are given by the coefficients of the \emph{plethysm} $\odot$ of complete symmetric homogeneous functions $h_k \odot h_n$ expressed in the basis of Schur functions (see~\cite[Ch. 5.4]{JamesKerber}). We abuse notation by letting $\lambda \vdash kn$ refer to its corresponding $S_{kn}$-irreducible. One can show that the $S_{kn}$-irreducible $(kn-2,2)$ has multiplicity 1 in the permutation representation of $S_{kn}$ on $\mathcal{M}_{kn}^k$ (see~\cite[Prop.~2.5(b)]{Weintraub90}, for example).  The coefficients of the plethysm $h_k \odot h_n$ expressed in the basis of Schur functions are seldom $\{0,1\}$-valued, and so the \emph{orbitals} of $S_{kn}$ acting on $\mathcal{M}_{kn}^k \times \mathcal{M}_{kn}^k$ can be realized as a set $\mathcal{A} = \{A_0,A_1,\cdots,A_d\}$ of binary matrices satisfying all the axioms of Definition~\ref{def:assoc} except for commutativity, i.e., $\mathcal{A}$ is a \emph{homogeneous coherent configuration}~\cite{Higman75}. For any $m,m' \in \mathcal{M}_{kn}^k$, let $d(m,m')$ be the \emph{meet table} of $m,m'$, that is, the $n \times n$ matrix with rows indexed by the $k$-edges $e \in m$, columns indexed by the $k$-edges $e' \in m'$, defined such that $d(m,m')_{e,e'} = |e \cap e'|$. It is well-known that the orbitals of $\mathcal{A}$ are in one-to-one correspondence with the isomorphism classes $\{\boldsymbol{\mu}\}$ of meet tables, which in turn are in bijection with double cosets $H_{k,n} \backslash S_{kn} / H_{k,n}$. The reason why $\mathcal{A}$ is not an association scheme is due to the fact that there are more non-isomorphic meet tables than there are isotypic components of $1 \!\! \uparrow_{H_{k,n}}^{S_{kn}}$. See~\cite[Ch.~15]{GodsilMeagher} for a detailed discussion of meet tables, homogeneous coherent configurations, and their connections to EKR combinatorics.

A \emph{bi-$(H_{k,n}\backslash S_{kn}/H_{k,n})$-invariant} function of the group algebra $\mathbb{R}[S_{kn}]$ is a function that is constant on the $(H_{k,n}\backslash S_{kn}/H_{k,n})$-double cosets.
To determine the orthogonal projection onto the space $(kn) \oplus (kn-2,2)$, it suffices to compute the unique norm-1 bi-$(H_{k,n}\backslash S_{kn}/H_{k,n})$-invariant function of $(kn-2,2) \leq \mathbb{R}[\mathcal{M}_{kn}^k] \leq \mathbb{R}[S_{kn}]$, i.e., the $(kn-2,2)$-\emph{spherical function} $\omega^{(kn-2,2)}$. 
To derive this spherical function, some baroque but standard representation-theoretic terminology will be needed, which can be found in~\cite{CST,Sagan}.

For any \emph{tabloid} $\{t\}$ of shape $(kn-2,2)$, define $\mathcal{S}(\{t\}) := \{ m \in \mathcal{M}_{kn}^k : \{t\}_2 \subseteq e \in m  \}$ where $\{t\}_2$ denotes the 2-set that is the second row of $\{t\}$. Define the \emph{intertwining map}
\[
	\mathcal{I}_{(kn-2,2)} : M^{(kn-2,2)} \rightarrow \mathbb{R}[\mathcal{M}_{kn}^k], \quad \mathcal{I}_{(kn-2,2)}(e_{\{t\}}) = 1_{\mathcal{S}(\{t\})}
\]
where $M^{(kn-2,2)}$ denotes the \emph{permutation module} with standard basis functions $e_{\{t\}}$. 
Recall that the orthogonal projection onto the space of $H_{k,n}$-invariant functions (as an element of the group algebra of $S_{kn}$) is as follows:
\[
	P = \frac{1}{|H_{k,n}|} \sum_{h \in H_{k,n}} h.
\]
Let $t$ be the standard Young tableau with bottom row $\{k+1,k+2\}$. Its \emph{column-stabilizer} is $C_t = \{ (), (1,k+1), (2,k+2), (1,k+1)(2,k+2)\}$. Let $e_t$ be the \emph{polytabloid} of $t$, i.e.,
\[
	e_t := \sum_{\sigma \in C_t} \text{sgn}(\sigma)~e_{\{\sigma t\}}.
\]
We have
\begin{align*}
	\omega^{(kn-2,2)} = \frac{1}{2} P\mathcal{I}_{(kn-2,2)}(e_t) = \frac{1}{2|H_{k,n}|} \sum_{h \in H_{k,n}}  &\mathcal{I}_{(kn-2,2)}(e_{h\{ t\}})
 + \mathcal{I}_{(kn-2,2)}(e_{h\{(1,3)(2,4) t\}})  \\
 &-   \mathcal{I}_{(kn-2,2)}(e_{h\{(1,3) t\}}) - \mathcal{I}_{(kn-2,2)}(e_{h\{(2,4) t\}}) .
\end{align*}
For each $\sigma \in C_t$, either the last row of $\{s\} := \{\sigma t\}$ is contained in some $e \in m^*$, or not. The positive terms fall into the former case, whereas the negative terms fall into the latter. 

For the former, we have
\begin{align*}
\frac{1}{|H_{k,n}|} \sum_{h \in H_{k,n}} \mathcal{I}_{(kn-2,2)}(e_{h\{ s\}}) &= \frac{1}{|H_{k,n}|}\sum_{h \in H_{k,n}} \mathbf{1}_{h\{s\}_2}\\
&= \frac{1}{\binom{k}{2}n}\sum_{\substack{\ell \subseteq e \in m^*  \\ |\ell | = 2} } \mathbf{1}_{\ell}\\
&= \sum_{ \boldsymbol{\mu} \in H_{k,n} \backslash S_{kn} / H_{k,n} } \!\!\!\! \frac{\text{fp}_{2/k}(\boldsymbol{\mu}) }{\binom{k}{2}n}  \mathbf{1}_{\boldsymbol{\mu}},
\end{align*}
where we define $\text{fp}_{2/k}(\boldsymbol{\mu}) := \sum_{e \in m} \sum_{e' \in m'}\binom{d(m,m')_{e,e'}}{2}$ for any $m \cup m' \cong \boldsymbol{\mu}$. 

For the latter, we have
\begin{align*}
\frac{1}{|H_{k,n}|} \sum_{h \in H_{k,n}} \mathcal{I}_{(kn-2,2)}(e_{h\{ s\}}) &= \frac{1}{|H_{k,n}|}\sum_{h \in H_{k,n}} \mathbf{1}_{h\{s\}_2}\\
&= \frac{1}{ \binom{kn}{2} - \binom{k}{2}n}\sum_{\substack{\ell \not \subseteq e \in m^*  \\ |\ell | = 2} } \mathbf{1}_{\ell}\\
&= \sum_{ \boldsymbol{\mu} \in H_{k,n} \backslash S_{kn} / H_{k,n} } \!\!\!\! \frac{\binom{k}{2}n-\text{fp}_{2/k}(\boldsymbol{\mu}) }{ \binom{kn}{2} - \binom{k}{2}n }  \mathbf{1}_{\boldsymbol{\mu}}.
\end{align*}
Let $\omega^{(kn-2,2)}_{\boldsymbol{\mu}}$ be the value of $\omega^{(kn-2,2)}$ at the double coset $\boldsymbol{\mu}$. Combining terms gives us
\[
	\omega^{(kn-2,2)}_{\boldsymbol{\mu}} = \frac{\text{fp}_{2/k}(\boldsymbol{\mu}) }{\binom{k}{2}n} - \frac{\binom{k}{2}n-\text{fp}_{2/k}(\boldsymbol{\mu}) }{ \binom{kn}{2} - \binom{k}{2}n } = \frac{1}{k} \left( \left( \frac{2}{(k-1)n} + \frac{2}{k(n-1)} \right)\text{fp}_{2/k}(\boldsymbol{\mu})  - \frac{(k-1)}{(n-1)} \right)
\]
for all $(H_{k,n} \backslash S_{kn} / H_{k,n})$-double cosets $\boldsymbol{\mu}$.
Through standard representation theory, the matrix representation $E_{1}$ of the orthogonal projection onto $(kn-2,2)$ can now be written as follows:
\begin{align*}
	(E_{1})_{ m,m'} &= \frac{\dim(kn-2,2)}{|\mathcal{M}_{kn}^k|} \sum_{m \in \mathcal{M}_{kn}^k} \omega^{(n-2,2)}_{d(m,m')}\\ 
	&= \left( \frac{\binom{kn}{2}-kn}{ |\mathcal{M}_{kn}^k|}\right) \left( \left( \frac{2}{k(k-1)n} + \frac{2}{k^2(n-1)} \right)\text{fp}_{2/k}(d(m,m'))  - \frac{(k-1)}{k(n-1)}\right)\\
	&=  \left( \frac{\binom{kn}{2}-kn}{ |\mathcal{M}_{kn}^k|} \right) \left(\left( \frac{2n+k-3}{k(k-1)n(n-1)} \right)\text{fp}_{2/k}(d(m,m'))  - \frac{(k-1)}{k(n-1)} \right)\\
	&=  \frac{1}{ |\mathcal{M}_{kn}^k|} \left(\left( \frac{(2n+k-3)(kn -3)}{2(k-1)(n-1)} \right)\text{fp}_{2/k}(d(m,m'))  - \frac{(k-1)}{k(n-1)}\left(\binom{kn}{2}-kn\right) \right)
\end{align*}
for all $m,m' \in \mathcal{M}_{kn}^k$. Double counting now gives the following proposition.
\begin{proposition}\label{prop:hyperE1}For any $\mathcal{F} \subseteq \mathcal{M}_{kn}^k$ and $m \in \mathcal{M}_{kn}^k$, we have
\[
(E_11_\mathcal{F})_{m} = \left( \frac{(2n+k-3)(kn -3)}{2(k-1)(n-1)|\mathcal{M}_{kn}^k|} \right)\!\! \sum_{ij \subseteq e \in m} \!\! |\mathcal{F}\!\!\downarrow_{ij}|  - \frac{|\mathcal{F}|(k-1)}{|\mathcal{M}_{kn}^k|k(n-1)}\left(\binom{kn}{2}-kn\right).
\]
\end{proposition}
\noindent Observe that setting $k=2$ recovers Proposition~\ref{prop:matchingE1}.

Because $(kn-2,2)$ has multiplicity 1, our expression for $\omega^{(kn-2,2)}$ allows us to easily read off the $(kn-2,2)$-eigenvalue of \emph{any} matrix in the non-commutative matrix algebra spanned by the orbitals. For example, we can write $\Gamma$ as the sum of all orbitals $A_{\boldsymbol{\mu}}$ such that $\text{fp}_{2/k}(\boldsymbol{\mu}) = 0$, in which case the $(kn-2,2)$-eigenvalue $\eta_{(kn-2,2)}$ of $\Gamma$ is simply
\[
	\eta_{(kn-2,2)} = \sum_{ \substack{  \boldsymbol{\mu} : \text{fp}_{2/k}(\boldsymbol{\mu}) = 0}} v_{\boldsymbol{\mu}}\omega^{(kn-2,2)}_{\boldsymbol{\mu}} = \frac{-D_{2/k}(k-1)}{k(n-1)}
\]
where $v_{\boldsymbol{\mu}}$ is the valency of $A_{\boldsymbol{\mu}}$ and $D_{2/k}$ is the degree of $\Gamma$, i.e., for a fixed $m^* \in \mathcal{M}_{kn}^k$, the number of $m \in \mathcal{M}_{kn}^k$ such that $|e \cap e^*| \leq 1$ for all $e \in m, e^* \in m^*$. In~\cite{MeagherPartial}, the eigenvalue $\eta_{(kn-2,2)}$ was computed by constructing a suitable \emph{equitable partition} of $\Gamma$~\cite{GodsilRoyle}, but such methods fall short of determining the $(kn-2,2)$-eigenvalue of any matrix in the span of $\mathcal{A}$. 

A family $\mathcal{F}$ is \emph{$ij$-centered} if $\mathcal{F}\!\!\downarrow_{ij} = \mathcal{F}$. 
Two families $\mathcal{G},\mathcal{G'}$ are \emph{cross partially 2-intersecting} if $m$ and $m'$ partially 2-intersect for all $m \in \mathcal{G}$, $m' \in \mathcal{G'}$. For any $ij$, let $\mathcal{F}_{ij} := \{ m \in \mathcal{M}_{kn}^k : ij \subseteq e \text{ for some } e \in m\}$ be the corresponding canonically partially 2-intersecting family. For the following combinatorial lemma, it will be useful to note that $|\mathcal{M}_{kn}^k| = \prod_{i = 0}^{n-1} \binom{k(n-i)}{k}  /n!$.

\begin{lemma}\label{lem:cross} Let $i,i',j,j' \in [kn]$ such that $ij \neq i'j'$. Let $\mathcal{G} \subseteq \mathcal{M}_{kn}^k$ be $ij$-centered, $\mathcal{G'} \subseteq \mathcal{M}_{kn}^k$ be $i'j'$-centered, and suppose that $\mathcal{G},\mathcal{G'}$ are cross partially 2-intersecting. Then
	$$|\mathcal{G}| \cdot |\mathcal{G}'| = o\left( \left(\binom{kn-2}{k-2}|\mathcal{M}_{k(n-1)}^k|\right)^2\right).$$
\end{lemma}
\begin{proof}
We just prove the case $|ij \cap i'j'| = 1$, since the other case $|ij \cap i'j'| = 0$ follows the same argument. We may assume $i = i'$ and $j' = \ell \neq j$. For any $ij$ and $X \subseteq [kn] - ij$ of size $k-2$, let $\mathcal{F}_{ijX} \subseteq \mathcal{F}_{ij}$ denote the members of $\mathcal{F}_{ij}$ that have the $k$-edge $ijX$. Then clearly we may write $\mathcal{F}_{ij}$ and $\mathcal{F}_{i \ell}$ as 
\[
	\mathcal{F}_{ij} = \bigsqcup_{X \in \binom{[kn]-ij}{k-2}} \mathcal{F}_{ijX} \quad \text{ and } \quad 	\mathcal{F}_{i\ell} = \bigsqcup_{Y \in \binom{[kn]-i\ell}{k-2}} \mathcal{F}_{i\ell Y},
\]
where each $\mathcal{F}_{ijX}$, $\mathcal{F}_{i\ell Y}$ has size $|\mathcal{M}_{k(n-1)}^k|$.
To prove the lemma, we give upper bounds on the total number of possible cross partially 2-intersecting (CP2I) pairs between any two partitions $(\mathcal{F}_{ijX},\mathcal{F}_{i\ell Y})$ of $\mathcal{F}_{ij}$ and $\mathcal{F}_{i\ell }$ ranging over all such $X$ and $Y$. Summing up these bounds over all pairs of partitions $(\mathcal{F}_{ijX},\mathcal{F}_{i\ell Y})$ gives an upper bound on the number of CP2I pairs between an $ij$-centered and an $i\ell$-centered family. We proceed in cases.
\begin{enumerate}
	\item $(\mathcal{F}_{ijX}, \mathcal{F}_{i\ell Y})$ such that $X \cap Y \neq \emptyset$ (thus $k \geq 3$): there are exactly $|\mathcal{M}_{k(n-1)}^k|^2$ CP2I pairs. The number of pairs of $k$-edges of the form $(ijX, i\ell Y)$ such that $X \cap Y \neq \emptyset$ is
	\[
	2\left( \binom{kn-3}{k-3} \cdot \binom{kn-2}{k-2} \right) + \binom{kn-3}{k-2} \cdot \left( \binom{kn-3}{k-2} - \binom{kn-3-(k-2)}{k-2} \right),	\]
	thus the total contribution of CP2I pairs of this type is $o((\binom{kn-2}{k-2}|\mathcal{M}_{k(n-1)}^k|)^2)$.
	\item $(\mathcal{F}_{ijX}, \mathcal{F}_{i\ell Y})$ such that $X \cap Y = \emptyset$: there are two types of CP2I pairs that can arise. We proceed in cases.
	\begin{itemize}
		\item A member $m \in \mathcal{F}_{i\ell Y}$ has a $k$-edge $e \in m$ such that $e \neq i\ell Y$ and $|e \cap (j+X)| \geq 2$ (thus $k \geq 3$). The number of CP2I pairs of this type cannot be more than
		$$|\mathcal{M}_{k(n-1)}^k| \binom{k-1}{2} \cdot  \binom{k(n-1)-2}{k-2} |\mathcal{M}_{k(n-2)}^k| = |\mathcal{M}_{k(n-1)}^k|^2O(1/n).$$ 
		Thus the total contribution of CP2I pairs of this type is $o((\binom{kn-2}{k-2}|\mathcal{M}_{k(n-1)}^k|)^2)$.
	
		\item A member $m \in \mathcal{F}_{i\ell Y}$ has $|e \cap ijX| \leq 1$ for all $e \in m$. There are clearly $o(|\mathcal{M}_{k(n-1)}^k|^2)$ CP2I pairs of this type, thus the total contribution of CP2I pairs of this type is $o((\binom{kn-2}{k-2}|\mathcal{M}_{k(n-1)}^k|)^2)$.
	\end{itemize}
\end{enumerate}
Summing up the total contributions of each type proves there are $o((\binom{kn-2}{k-2}|\mathcal{M}_{k(n-1)}^k|)^2)$ CP2I pairs between any $ij$-centered family and $i\ell$-centered family, as desired.
\end{proof}

Let $A_1 \in \mathcal{A}$ be the orbital corresponding to the meet table with $n-2$ nonzero diagonal entries, i.e., $(A_1)_{m,m'} = 1$ if and only if there is a transposition $\tau \in S_{kn}$ such that $\tau m = m'$ and $m \neq m'$. In this case, we say that $m$ and $m'$ are 1-related. The graph of $A_1$ is connected. A \emph{vertex-separator} of a graph is a subset of vertices whose removal disconnects the graph.

We are now in a position to prove the main conjecture of~\cite{MeagherPartial}. Note that the condition of $n$ being sufficiently large stems from the proof of~\cite[Theorem 6.1]{MeagherPartial}, so we have made no attempt to sharpen our result to hold for all $n$.
\begin{theorem}\emph{\cite[Conjecture 6.1]{MeagherPartial}}\label{thm:partial}
	If $\mathcal{F} \subseteq \mathcal{M}_{kn}^k$ is a largest partially 2-intersecting family, then $\mathcal{F}$ is a canonically partially 2-intersecting family provided that $n$ is sufficiently large.
\end{theorem}
\begin{proof}
Let $P = E_0 + E_1$ be the orthogonal projection onto $V_0 \oplus V_1$.
Since $A_1$ is connected, let $m_1 \in \mathcal{F}$ and $m_0 \notin \mathcal{F}$ be 1-related. 
Let $P_m := (P1_{\mathcal{F}})_m$. In the proof of~\cite[Theorem 6.1]{MeagherPartial} Meagher et al. showed that $P_m = (1_{\mathcal{F}})_m$, so by Proposition~\ref{prop:hyperE1} we have
\[
P_{m_1} - P_{m_0} = \left( \frac{(2n+k-3)(kn -3)}{2(k-1)(n-1)|\mathcal{M}_{kn}^k|} \right)\left[ \sum_{ij \subseteq e \in m_1} |\mathcal{F}\!\!\downarrow_{ij}|  - \sum_{ij \subseteq e \in m_0} |\mathcal{F}\!\!\downarrow_{ij}| \right] = 1.
\]
Let $S_1$ be the set of 2-sets contained in a $k$-edge of $m_1$, and similarly $S_0$ for $m_0$. Let $X_1 := S_1 \setminus S_0$ and $X_0 := S_0 \setminus S_1$. Note that $|X_0| = |X_1| = 2(k-1)$. We may write 
\[
	\sum_{ij \in X_1} |\mathcal{F}\!\!\downarrow_{ij}|  - \sum_{ij \in X_0} |\mathcal{F}\!\!\downarrow_{ij}|  = 	\left( \frac{2(k-1)(n-1)}{(2n+k-3)(kn -3)} \right)|\mathcal{M}_{kn}^k|.
\]
After dropping negative terms, averaging shows there exists a 2-set $ij \in X_1$ such that
\[
	|\mathcal{F}\!\!\downarrow_{ij}| \geq 	\left( \frac{(n-1)}{(2n+k-3)(kn -3)} \right)|\mathcal{M}_{kn}^k| = \Omega\left( \binom{kn-2}{k-2}|\mathcal{M}_{k(n-1)}^k| \right).
\]
Suppose that $\mathcal{F} \setminus \mathcal{F}\!\!\downarrow_{ij}$ is nonempty. Then we claim the following.\\

\noindent \textbf{Claim 2:} There exists a 2-set $i'j' \neq ij$ such that $|\mathcal{F}\!\!\downarrow_{i'j'}\!\!| = \Omega\left( \binom{kn-2}{k-2}|\mathcal{M}_{k(n-1)}^k| \right)$.
\begin{proof}[Proof of Claim 2]
Suppose there is no pair $m' \in \mathcal{F} \! \setminus \! \mathcal{F}\!\! \downarrow_{ij}$, $m \in \mathcal{F}\!\! \downarrow_{ij}$ that are 1-related, that is, $\mathcal{M}_{kn}^k \!\!\setminus \! \mathcal{F}$ forms a vertex-separator of $A_1$ that separates $\mathcal{F} \!\! \downarrow_{ij}$ from $\mathcal{F}\! \setminus \! \mathcal{F} \!\!\downarrow_{ij}$. Since $A_1$ is connected, there is another 1-related pair $w_1 \in \mathcal{F} \setminus \mathcal{F}\downarrow_{ij}$, $w_0 \notin \mathcal{F}$ such that $w_1 \neq m_1$. We may write
\[
\sum_{i'j' \in Y_1} |\mathcal{F}\!\!\downarrow_{i'j'} \!\! |  - \sum_{i'j' \in Y_0} |\mathcal{F}\!\!\downarrow_{i'j'}\!\!| = 	\left( \frac{2(k-1)(n-1)}{(2n+k-3)(kn -3)} \right)|\mathcal{M}_{kn}^k|
\]
where $Y_1$ is the collection of 2-sets that are contained in a $k$-edge of $w_1$ but not a $k$-edge of $w_0$, and vice versa for $Y_0$. Since $w_1 \in \mathcal{F} \setminus \mathcal{F} \downarrow_{ij}$, we have $ij \notin Y_1$, and so we may repeat the argument preceding the claim to obtain another 2-set $i'j' \neq ij$ such that
\[
|\mathcal{F}\!\!\downarrow_{i'j'} \!\!| = \Omega \left( \binom{kn-2}{k-2}|\mathcal{M}_{k(n-1)}^k| \right).
\]

Now suppose there exists a $m' \in \mathcal{F} \! \setminus \! \mathcal{F} \!\! \downarrow_{ij}$ and $m \in \mathcal{F}\!\!\downarrow_{ij}$ that are 1-related. Let $Z_1$ be the collection of 2-sets that are contained in a $k$-edge of $m$ but not a $k$-edge of $m'$, and vice versa for $Z_0$. Since $P_m - P_{m'} = 0$, Proposition~\ref{prop:hyperE1} implies that
\[
\sum_{i'j' \in Z_1} |\mathcal{F}\!\!\downarrow_{i'j'} \!\! |   = \sum_{i'j' \in Z_0} |\mathcal{F}\!\!\downarrow_{i'j'} \!\! |. 
\]
Since $ij \in Z_1$, we have $$
	| \mathcal{F} \!\! \downarrow_{ij} \!\!| \leq  \sum_{i'j' \in Z_0} |\mathcal{F}\!\!\downarrow_{i'j'}\!\!|.$$ 
Since $ij \notin Z_0$, averaging shows there exists a 2-set $i'j' \neq ij$ such that 
\[
	\frac{| \mathcal{F}\!\! \downarrow_{ij}\!\! |}{2(k-1)} \leq |\mathcal{F}\!\!\downarrow_{i'j'}\!\! | = \Omega \left( \binom{kn-2}{k-2}|\mathcal{M}_{k(n-1)}^k| \right).
\]
This completes the proof of the claim.
\end{proof}
\noindent Since $\mathcal{F}$ is partially 2-intersecting, $\mathcal{F} \!\! \downarrow_{ij}$, $\mathcal{F}\!\! \downarrow_{i'j'}$ are cross partially 2-intersecting, thus
$$|\mathcal{F} \!\! \downarrow_{ij} \!\! | \cdot |\mathcal{F}\!\! \downarrow_{i'j'}\!\! | = \Omega \left( \left( \binom{kn-2}{k-2}|\mathcal{M}_{k(n-1)}^k| \right)^2 \right),$$ 
which contradicts Lemma~\ref{lem:cross}. We deduce that $\mathcal{F}$ is a largest partially 2-intersecting family that is centered, i.e., $\mathcal{F}$ is a canonically partially 2-intersecting family.
\end{proof}
\noindent We now discuss why the algebraic techniques for proving uniqueness mentioned in the introduction seem ill-suited for partially 2-intersecting families. Let $M$ be the $ |\mathcal{M}_{kn}^k| \times \binom{kn}{2} $ matrix whose columns are the characteristic vectors of the canonically partially 2-intersecting families of $\mathcal{M}_{kn}^k$. Label the columns of $M$ with edges $ij \in E(K_{kn})$ so that each row of $M$ is the characteristic vector $1_F \in \mathbb{R}^{E(K_{kn})}$ of a \emph{$K_k$-factor} $F$ of $K_{kn}$, i.e., a set $F \subseteq E(K_{kn})$ whose edges form $n$ vertex-disjoint $k$-cliques that cover the vertices of $K_{kn}$. Let $c_G \in \mathbb{R}^{E(K_{kn})}$ be the characteristic vector of a $kn$-vertex graph $G$, so that
$$ \max c_G^\top x \quad \text{subject to } x \in \text{conv.hull}(rows(M)) =: P \quad \leq \quad n \cdot k(k-1)/2,$$
equality holding if and only if $G$ has a $K_k$-factor. However, for $k \geq 3$, deciding if a graph has a $K_k$-factor is NP-complete~\cite{HellK83}. In light of this, it is unlikely that good descriptions of the facet-inducing inequalities of $P$ exist, which seems necessary to get started using the polyhedral method. As for dual width, it is not even clear what this means in the context of homoegeneous coherent configurations. Finally, it might be possible to apply the rank method to $M$, but a shorter proof along these lines seems unlikely, even for $k=2$.

\section{Conclusion and Open Questions}

The proof technique in this work can be seen as a \emph{primal-dual method}: the primal eigenvalues $P_d(i)$ of the top associate $A_d$ give the upper bound whereas the dual eigenvalues $Q_1(j)$ of least non-trivial idempotent $E_1$ characterize the case of equality. However, in Section~\ref{sec:hypermatchings}, Meagher et al.'s proof of the bound uses the adjacency matrix of the partial 2-derangement graph (a union of orbitals) rather than a single ``top orbital". It would be interesting to see if a strengthening of their main result holds by showing there exists a single orbital  $A_{\boldsymbol{\mu}}$ with $\text{fp}_{2/k}(\boldsymbol{\mu}) = 0$ that meets the ratio bound with equality.

We have seen that this method works well for combinatorial objects; however, some of the $q$-analogues of these objects correspond to structured matrices with entries in $\mathbb{F}_q$ that seem less receptive to the method. For example, we were only able to prove the EKR theorem for bilinear forms provided that $m \leq \lceil n/2 \rceil$, which is surprising because the bilinear forms association scheme has an algebraic structure quite similar to the Johnson and Grassmann schemes. The situation is even worse for matrix groups such as $\{\mathit{GL}(n,q)\}_{n=1}^\infty$, but here we are able to offer a somewhat formal explanation as to why our techniques fall short.

In~\cite{DafniFLLV21}, the \emph{chunk size} $\mathfrak{c}$ of a domain $\mathcal{X}$ is defined to be the following quantity:
$$\mathfrak{c} := \min_{\substack{ x,y \in \mathcal{X} \\ x \neq y} } |x \setminus y|.$$
For example, $\mathfrak{c} = 1$ if $\mathcal{X}$ is the collection of $k$-element subsets of $[n]$ or words of $[q]^n$, $\mathfrak{c} = 2$ if $\mathcal{X}$ is the symmetric group or the set of perfect matchings, and $\mathfrak{c} = 3$ if $\mathcal{X}$ is the alternating group. However, there are natural choices of group domains that have much larger chunk size. Take for example $\mathcal{X} = \mathit{GL}(n,q)$ where we represent the invertible linear map $g \in \mathit{GL}(n,q)$ naturally as a subset of ordered pairs $\Sigma = V \times V$ where $V$ is the set of 1-dimensional subspaces of $\mathbb{F}_q^n$. Note that this is just the standard representation of $\mathit{GL}(n,q)$ as a group of $[n]_q \times [n]_q$ permutation matrices. It is not difficult to see that $\mathfrak{c} = \Theta([n]_q)$ for $\mathit{GL}(n,q)$, $\mathit{SL}(n,q)$, and their projective versions. This poses a problem for the presented proof method due to the fact that there is no element of any of these groups that moves $O(1)$ points of $\Sigma$. More generally, the \emph{minimal degree} $\mu(G_n)$ of a permutation group $G_n$ acting on $n$ points is defined to be the smallest number of moved points of a non-identity element of $G_n$. One of the classical problems in permutation groups is to classify the permutation groups whose minimal degree is ``small". There are many results that show if the minimal degree is small, then the group contains the alternating group $A_n$, e.g., if $A_n \not \leq G_n$, then $\mu(G_n) \geq (\sqrt{n}-1)/2$ (see~\cite{Babai81}). Symmetric groups, and more generally Coxeter groups, are important infinite families that are not ruled out by this result, but generally speaking, there seem to be relatively few natural families of finite groups with a large alternating subgroup.

Finally, it would be interesting to see if these primal-dual methods can provide shorter uniqueness proofs for \emph{t-intersecting} EKR results, where two elements are said to \emph{$t$-intersect} if they have $t$ or more atoms in common.  For cometric association schemes, width methods have been used to characterize maximum $t$-intersecting families, and for the Johnson scheme a polyhedral proof is known (see~\cite{GodsilMeagher}). Using our methods, we have only been able to obtain unsatisfactory results for $k$-sets of $[n]$ for constant $k,t$ and $n \gg (t+1)(k-t-1)$. We leave it as an open question whether these techniques can be pushed further to give sharper results.

\bibliographystyle{plain}
\bibliography{../research/master.bib}

\end{document}